\documentclass{amsart}

\usepackage{amsthm,amssymb,amsmath,amscd}

\theoremstyle{plain}
\newtheorem{thm}{Theorem}[section]
\newtheorem{lem}[thm]{Lemma}
\newtheorem{prop}[thm]{Proposition}

\theoremstyle{definition}
\newtheorem*{defn}{Definition}
\newtheorem*{ex}{Example}

\theoremstyle{remark}
\newtheorem*{rem}{Remark}
\newtheorem*{ack}{Acknowledgement}

\title[C$^*$-algebras generated by composition operators]
{C$^*$-algebras generated by multiplication operators and
composition operators with self-similar maps}
\author{Hiroyasu Hamada}

\address{National Institute of Technology (KOSEN), Sasebo College, 
Okishin, Sasebo, Nagasaki, 857-1193, Japan.}
\email{h-hamada@sasebo.ac.jp}



\begin{document}

\begin{abstract}
Let $K$ be a compact metric space and
let $\gamma = (\gamma_1, \dots, \gamma_n)$
be a system of proper contractions on $K$. 
We study a C$^*$-algebra $\mathcal{MC}_{\gamma_1, \dots, \gamma_n}$
generated by
all multiplication operators by continuous functions on $K$
and composition operators $C_{\gamma_i}$
induced by $\gamma_i$ for $i=1, \dots, n$ on
a certain $L^2$ space.
Suppose that $K$ is self-similar.
We consider the Hutchinson
measure $\mu^H$ of $\gamma$ and the $L^2$ space $L^2(K, \mu^H)$.
Then we show
that the C$^*$-algebra $\mathcal{MC}_{\gamma_1, \dots, \gamma_n}$
is isomorphic to
the Cuntz algebra $\mathcal{O}_n$ under some conditions.
\end{abstract}

\maketitle

\section{Introduction}

Several authors considered
C$^*$-algebras generated by composition operators
(and Toeplitz operators) on the Hardy space $H^2(\mathbb{D})$ on the open
unit disk
$\mathbb{D}$ (\cite{H1, HW, J1, J2, KMM1, KMM3, KMM2, Pa, Q, Q2, SA}).
On the other hand, there are some studies on C$^*$-algebras generated by
composition operators on $L^2$ spaces, for example \cite{H2, H3, M}.
Matsumoto \cite{M} introduced some
C$^*$-algebras associated with cellular automata generated by composition
operators and multiplication operators.
Let $R$ be a rational function of degree at least two,
let $J_R$ be the Julia set of $R$ and let $\mu^L$ be the Lyubich measure
of $R$. In \cite{H2}, we studied the C$^*$-algebra $\mathcal{MC}_R$
generated by
all multiplication operators by continuous functions in $C(J_R)$
and the composition operator $C_R$ induced by $R$
on $L^2(J_R, \mu^L)$.

Let $K$ be a compact metric space, 
let $\gamma = (\gamma_1, \dots, \gamma_n)$ be a system of proper contractions
on $K$ and let $\varphi: K \to K$ be measurable.
Suppose that $\gamma_1, \dots, \gamma_n$ are inverse branches of
$\varphi$ and $K$ is self-similar. Let $\mu^H$ be the Huchinson measure of
$\gamma$. 

In \cite{H3}, we studied the C$^*$-algebra $\mathcal{MC}_\varphi$
generated by
all multiplication operators by continuous functions in $C(K)$
and the composition operator $C_\varphi$ induced by $\varphi$
on $L^2(K, \mu^H)$.
Then the C$^*$-algebra $\mathcal{MC}_\varphi$ is isomorphic to
the C$^*$-algebra $\mathcal{O}_\gamma (K)$ associated
with $\gamma$ introduced in \cite{KW} under some condition.

In this paper we study
a C$^*$-algebra $\mathcal{MC}_{\gamma_1, \dots, \gamma_n}$
generated by all multiplication operators by continuous functions in $C(K)$
and composition operators $C_{\gamma_i}$
induced by $\gamma_i$ for $i=1, \dots, n$ on $L^2(K, \mu^H)$.
Then we show
that the C$^*$-algebra $\mathcal{MC}_{\gamma_1, \dots, \gamma_n}$
is isomorphic to
the Cuntz algebra $\mathcal{O}_n$ under some conditions.

We can prove the main theorem using Cuntz-Pimsner algebras
and a proposition in \cite{PWY}. However we prove this theorem
directly using multiplication operators and composition operators
in this paper.

\section{Main Theorem}

Let $(K,d)$ be a compact metric space. A continuous map $\gamma: K \to K$
is called a {\it proper contraction} if there exists constants
$0 < c_1 \leq c_2 < 1$ such that
\[
   c_1 d(x, y) \leq d(\gamma(x), \gamma(y)) \leq c_2 d(x, y), \quad x, y \in K.
\]

Let $\gamma = (\gamma_1, \dots, \gamma_n)$ be a family of proper contractions
on $(K,d)$. We say that $K$ is called {\it self-similar}
with respect to $\gamma$ if $K = \bigcup_{i = 1} ^n \gamma_i (K)$.
See \cite{F} and \cite{Ki} for more on fractal sets.

Let us denote by $\mathcal{B}(K)$ the Borel $\sigma$-algebra on $K$.

\begin{lem}[\cite{Hu}] \label{lem:Hutchinson measure}
Let $K$ be a compact metric space and let $\gamma$ be a system of proper
contractions. If $p_1, \dots, p_n \in \mathbb{R}$ satisfy
$\sum_{i=1} ^n p_i = 1$ and $p_i > 0$ for $i$,
then there exists a unique probability measure $\mu$ on $K$ such that
\[
   \mu(E) = \sum_{i=1}^n  p_i \mu (\gamma_i ^{-1}(E))
\]
for $E \in \mathcal{B}(K)$.
\end{lem}

We call the measure $\mu$ given by Lemma \ref{lem:Hutchinson measure}
the {\it self-similar measure} on $K$ with $\{p_i \}_{i = 1} ^n$.
In particular, we denote by $\mu^H$ the self-similar measure with
$p_i = \frac{1}{n}$ for $i$ and call this measure the {\it Hutchinson measure}.
We say that $\gamma$ satisfies the {\it measure separation condition} in $K$
if $\mu (\gamma_i (K) \cap \gamma_j (K)) = 0$ for any self-similar measure $\mu$and $i \neq j$.

For $a \in L^\infty (K, \mathcal{B}(K), \mu^H)$, we define
the multiplication operator $M_a$ on
$L^2 (K, \mathcal{B}(K), \mu^H)$ by
$M_a f = a f$ for $f \in L^2 (K, \mathcal{B}(K), \mu^H)$.
Let $\varphi : K \to K$ be measurable.
Suppose that $\gamma_1, \dots, \gamma_n$ are inverse
branches of $\varphi$, that is, $\varphi (\gamma_i (x)) = x$
for $x \in K$ and $i = 1, \dots, n$.

Let $i = 1, \dots, n$.
Set $({\gamma_i}_* \mu^H) (E)
= \mu^H (\gamma_i ^{-1} (E))$ and
 $(\mu^H \circ \gamma_i) (E) = \mu^H(\gamma_i(E))$
for $E \in \mathcal{B}(K)$.
Then ${\gamma_i}_* \mu^H$ and $\mu^H \circ \gamma_i$ are measures on $K$.

\begin{lem}  \label{lem:measure}
Let $\gamma = (\gamma_1, \dots, \gamma_n)$ be a system of proper contractions
on $K$. Assume that $K$ is self-similar and
the system $\gamma = (\gamma_1, \dots, \gamma_n)$ satisfies
the measure separation condition in $K$.
Then we have
\[
  (\mu^H \circ \gamma_i)(E) = \frac{1}{n} \mu^H(E)
\]
for $E \in \mathcal{B}(K)$.
\end{lem}

\begin{proof}
By the definition of Hutchinson measure, we have
\[
   \mu^H (E) = \frac{1}{n} \sum_{i=1} ^n \mu^H (\gamma_i ^{-1}(E))
\]
for $E \in \mathcal{B}(K)$.
Since $\gamma_i$ is a proper contraction, $\gamma_i: K \to \gamma_i(K)$ is
bijective. Thus we obtain the desired equation.
\end{proof}

We can calculate the Radon-Nikodym derivative
$\frac{d {\gamma_i}_* \mu^H}{d \mu^H}$ in the following lemma.

\begin{lem} \label{lem:Radon-Nikodym}
Let $\gamma = (\gamma_1, \dots, \gamma_n)$ be a system of proper contractions
on $K$. Assume that $K$ is self-similar and
the system $\gamma = (\gamma_1, \dots, \gamma_n)$ satisfies
the measure separation condition in $K$.
Then we have
\[
   \frac{d {\gamma_i}_* \mu^H}{d \mu^H} (x)
   = \begin{cases}
      n \quad & \text{if} \quad x \in \gamma_i (K), \\
      0 \quad & \text{if} \quad x \notin \gamma_i (K)
   \end{cases}
\]
for $i = 1, \dots, n$.
\end{lem}

\begin{proof}
Fix $E \in \mathcal{B}(K)$ and $i \neq j$.
We consider $F := \gamma_i ^{-1}(\gamma_j (E))$.
Since the system $\gamma = (\gamma_1, \dots, \gamma_n)$ satisfies
the measure separation condition in $K$, we have
\[
   {\gamma_i}_* \mu^H (\gamma_j(E)) = \mu^H (F)
   = n \mu^H (\gamma_i(F))
   \leq n \mu^H (\gamma_i(K) \cap \gamma_j(K)) = 0
\]
by Lemma \ref{lem:measure}.
Thus ${\gamma_i}_* \mu^H (\gamma_j(E)) = 0$.
It folows that
\begin{align} \label{eq:(1)}
\frac{d {\gamma_i}_* \mu^H}{d \mu^H} (x) = 0
\end{align} 
for $x \notin \gamma_i (K)$.
On the other hand, by Lemma \ref{lem:Hutchinson measure}, we have
\begin{align} \label{eq:(2)}
\sum_{i = 1} ^n \frac{d {\gamma_i}_* \mu^H}{d \mu^H} = n.
\end{align}
From (\ref{eq:(1)}) and (\ref{eq:(2)}), we obtain the desired conclusion.
\end{proof}

Write $C_{\gamma_i} f = f \circ \gamma_i$ for
$f \in L^2(K, \mathcal{B}(K), \mu^H)$.
Assume that $K$ is self-similar and
the system $\gamma = (\gamma_1, \dots, \gamma_n)$ satisfies
the measure separation condition in $K$.
By Lemma \ref{lem:Radon-Nikodym},
the Radon-Nikodym derivative 
$\frac{d {\gamma_i}_* \mu^H}{d \mu^H}$ is bounded.
By \cite[Theorem 2.1.1]{SM}, the operator
$C_{\gamma_i}: L^2( K , \mathcal{B}(K), \mu^H) \to
L^2( K , \mathcal{B}(K), \mu^H)$ is bounded.
It is called a composition  operator on $L^2( K , \mathcal{B}(K), \mu^H)$
induced by $\gamma_i$.
Set $V_i = \frac{1}{ \sqrt{n}} C_{\gamma_i} ^*$.

\begin{prop} \label{prop:Cuntz_relation}
Let $\gamma = (\gamma_1, \dots, \gamma_n)$ be a system of proper contractions
on $K$ and let $\varphi: K \to K$ be measurable.
Suppose that $\gamma_1, \dots, \gamma_n$ are inverse branches of
$\varphi$. Assume that $K$ is self-similar and
the system $\gamma = (\gamma_1, \dots, \gamma_n)$ satisfies
the measure separation condition in $K$.
Then we have
\[
   V_i ^* V_i = I \quad \text{and} \quad \sum_{j = 1} ^n V_j V_j ^* = I
\]
for $i = 1, \dots, n$.
\end{prop}

\begin{proof}
We consider the closed subspace
$\mathcal{K}_i = \{ f \in L^2(K, \mathcal{B}(K), \mu^H) \, | \,  f \, \, \text{vanishes on} \, K \smallsetminus \gamma_i(K) \}$
of $L^2(K, \mathcal{B}(K), \mu^H)$.
Since the system $\gamma = (\gamma_1, \dots, \gamma_n)$ satisfies
the measure separation condition in $K$, we have
$\mathcal{K}_i \perp \mathcal{K}_j$ for $i \neq j$. It is
sufficient to show that
$V_i ^*  : \mathcal{K}_i \to L^2(K, \mathcal{B}(K), \mu^H)$ is a
surjective isometry.
By Lemma \ref{lem:Radon-Nikodym}, we have
\begin{align*}
\| V_i ^* f \| ^2
& = \frac{1}{n} \int_K | f(\gamma_i (x)) | ^2 d \mu^H (x)
  = \frac{1}{n} \int_K |f(y)|^2 \, d({\gamma_i}_* \mu^H)(y) \\
& = \int_{\gamma_i(K)} |f(y)|^2 \, d \mu^H (y)
  = \int_K |f(y)|^2 \, d \mu^H (y)
  = \| f \|
\end{align*}
for $f \in \mathcal{K}_i$, where $y = \gamma_i (x)$ is a change of variables.
Thus $V_i ^*$ is an isometry.

For $g \in L^2(K, \mathcal{B}(K), \mu^H)$, we define the function
$f: K \to \mathbb{C}$ by
\[
   f(y) =
   \begin{cases}
   \sqrt{n} \, g(\varphi(y)) \quad & \text{if} \quad y \in \gamma_i (K), \\
   0 \quad & \text{if} \quad y \notin \gamma_i (K).
   \end{cases}
\]
Since $\gamma_{i}: K \to \gamma_i(K)$ is bijective
and its inverse is $\varphi|_{\gamma_i(K)}: \gamma_i(K) \to K$, we have
\begin{align*}
\int_K |f(y)|^2 \, d \mu^H(y)
&= n \int_{\gamma_i (K)} |g(\varphi(y))|^2 \, d \mu^H (y) \\
&= n \int_K |g(x)|^2 \, d(\mu^H \circ \gamma_i)(x)
= \int_K |g(x)|^2 \, d \mu^H (x)
\end{align*}
by Lemma \ref{lem:measure}.
Therefore it follows that $f \in \mathcal{K}_i$.
We also have
\[
   (V_i ^* f) (x) = \frac{1}{\sqrt{n}} \, f(\gamma_i(x))
                  = g(\varphi(\gamma_i(x)))
                  = g(x)
\]
for $x \in K$.
Hence $V_i ^*$ is surjective.
\end{proof}

\begin{defn}
We denote by $\mathcal{MC}_{\gamma_1, \dots, \gamma_n}$
the C$^*$-algebra generated by all multiplication operators by continuous
functions in $C(K)$ and composition operators
$C_{\gamma_i}$ by $\gamma_i$ for $i = 1, \dots, n$
on $L^2 (K, \mathcal{B}(K), \mu^H)$.
\end{defn}

The following theorem is the main result of the paper.

\begin{thm} \label{thm:main}
Let $K$ be a compact metric space, 
let $\gamma = (\gamma_1, \dots, \gamma_n)$ be a system of proper contractions
on $K$ and let $\varphi: K \to K$ be measurable.
Suppose that $\gamma_1, \dots, \gamma_n$ are inverse branches of
$\varphi$. Assume that $K$ is self-similar and
the system $\gamma = (\gamma_1, \dots, \gamma_n)$ satisfies
the measure separation condition in $K$.
Then $\mathcal{MC}_{\gamma_1, \dots, \gamma_n}$ is isomorphic to the Cuntz
algebra $\mathcal{O}_n$.
\end{thm}

\begin{proof}
By Proposition \ref{prop:Cuntz_relation},
the C$^*$-subalgebra $C^* (V_1, \dots, V_n)$ generated by $V_1, \dots, V_n$
in $\mathcal{MC}_{\gamma_1, \dots, \gamma_n}$ is isomorphic to the Cuntz
algebra $\mathcal{O}_n$. We only need to show that
$M_a \in C^* (V_1, \dots, V_n)$ for $a \in C(K)$.

For a finite word
$\omega = (\omega_1, \dots, \omega_k) \in \{1, \dots, n \}^k$,
we use the multi-index notation $V_\omega = V_{\omega_1} \dots V_{\omega_k}$
and $\gamma_\omega = \gamma_{\omega_1} \circ \dots \circ \gamma_{\omega_k}$.
Fix $x_0 \in K$.
For distinct finite words
$\omega(1), \dots, \omega(n^k) \in \{1, \dots, n \}^k$, we define
operators $V$ and $D$ on $\bigoplus_{i = 1} ^{n^k} L^2 (K, \mathcal{B}(K), \mu^H)$ by
\[
  V = \begin{pmatrix} V_{\omega(1)} & \cdots & V_{\omega(n^k)} \\
  0 & \cdots & 0 \\ \vdots & \ddots & \vdots \\ 0 & \cdots & 0
  \end{pmatrix}
\]
and
\[
  D = \begin{pmatrix}
  M_{a \circ \gamma_{\omega(1)} - (a \circ \gamma_{\omega(1)}) (x_0)}
  & 0 & \cdots & 0 \\
  0 & M_{a \circ \gamma_{\omega(2)} - (a \circ \gamma_{\omega(2)}) (x_0)}
  & \cdots & 0 \\
  \vdots & \vdots & \ddots & \vdots \\
  0 & 0 & \cdots & M_{a \circ \gamma_{\omega(n^k)}  - (a \circ \gamma_{\omega(n^k)}) (x_0)}
  \end{pmatrix}.
\]
From Proposition \ref{prop:Cuntz_relation}, it follows that
$\sum_{i = 1} ^{n^k} V_{\omega(i)} V_{\omega(i)} ^* = I$ and
$V^* V = I$. Thus we have
\begin{align*}
& \left\| M_a - \sum_{i = 1} ^{n^k} (a \circ \gamma_{\omega(i)}) (x_0)
   V_{\omega(i)} V_{\omega(i)} ^* \right\| \\
& = \left\| M_a \sum_{i = 1} ^{n^k} V_{\omega(i)} V_{\omega(i)} ^*
   - \sum_{i = 1} ^{n^k} (a \circ \gamma_{\omega(i)}) (x_0)
   V_{\omega(i)} V_{\omega(i)} ^* \right\| \\
& = \left\| \sum_{i = 1} ^{n^k} V_{\omega(i)} M_{a \circ \gamma_{\omega(i)}
   - (a \circ \gamma_{\omega(i)}) (x_0)} V_{\omega(i)} ^* \right\| \\
& = \| V D V^* \| \\
& = \| D \| \\
& = \max{ \left\{ \| a \circ \gamma_{\omega(i)} - (a \circ \gamma_{\omega(i)}) (x_0) \|
    \, | \, i = 1, \dots, n^k \right\}}.
\end{align*}
Since $a$ is uniformly continuous
and $\gamma_1, \dots, \gamma_n$ are proper contractions,
the above norm is sufficient small for a sufficient large $k$.
Thus $M_a \in C^* (V_1, \dots, V_n)$, which completes the proof.
\end{proof}

\begin{rem}
In the proof of Theorem \ref{thm:main},
we mainly use multiplication operators and
composition operators themselves.
We can also prove this theorem using the Cuntz-Pimsner algebras
as following.

Consider the C$^*$-algebra $A = C(K)$ and the canonical Hilbert
right $A$-module $X = A^n$. For $i = 1, \dots, n$, we define
a left $A$-action $\phi : A \to \mathcal{L}_A (X)$ by the diagonal matrix
\[
   (\phi (a) ) (x) =  {\rm diag} (a(\gamma_1(x)), \dots, a(\gamma_n(x)))
\]
for $a \in A, \, x \in K$, where we identify $\mathcal{L}_A (X)$ with
$M_n (A)$.
We consider the Cuntz-Pimsner algebra $\mathcal{O}_X$ of the Hilbert
bimodule $X$ over $A$. We note that $\mathcal{O}_X$ is the universal
C$^*$-algebra generated by $a \in A$ and $S_\xi$ with $\xi \in X$
satisfying that $a S_\xi = S_{\phi(a) \xi}, \, S_\xi a = S_{\xi a}, \,
S_\xi S_\eta = \langle \xi, \eta \rangle_A$ for $a \in A, \, \xi, \eta \in X$
and $\sum_{i = 1} ^n S_{u_i} S_{u_i} ^* = 1$, where $\{ u_i \}_{i = 1} ^n$
is the canonical finite basis of $X$.
By Remark 4.8 or Proposition 4.10 in \cite{PWY}, the C$^*$-algebra $\mathcal{O}_X$ is canonically isomorphic to the Cuntz algebra $\mathcal{O}_n$. Hence
$\mathcal{O}_X$ is simple.

It is easily seen that $a S_{u_i} = S_{u_i} (a \circ \gamma_i)$ and $S_{u_i}^* S_{u_i} = 1$ for $a \in A$ and $i = 1, \dots, n$.
On the other hand, we have corresponding equations $M_a V_i = V_i (a \circ \gamma_i), \, V_i^* V_i = I$ and $\sum_{j = 1} ^n V_j V_j ^* = I$ for $a \in A$ and $i = 1, \dots, n$ by Proposition \ref{prop:Cuntz_relation}.
By the universality and the simplicity of $\mathcal{O}_X$, there exists an
isomorphism $\Phi: \mathcal{O}_X \to \mathcal{MC}_{\gamma_1, \dots, \gamma_n}$
such that $\Phi(a) = M_a, \, \Phi(S_{u_i}) = V_i$ for $a \in A$ and $i = 1,
\dots, n$. Therefore the C$^*$-algebra $\mathcal{MC}_{\gamma_1, \dots, \gamma_n}$ is isomorphic to the Cuntz algebra $\mathcal{O}_n$.
\end{rem}

We give some examples for C$^*$-algebras generated by composition
operators $C_{\gamma_1}, \dots, C_{\gamma_n}$ and multiplication operators.

\begin{ex}
Let $K = [0,1]$ and
\[
  \gamma_1 (x) = \frac{1}{2} x, \quad
  \gamma_2 (x) = \frac{1}{2} x + \frac{1}{2}.
\]
Then $K$ is the self-similar set with respect to
$\gamma = (\gamma_1, \gamma_2)$. 
The Hutchinson measure $\mu^H$ on $K$ coincides with
the Lebesgue measure on $K$. 
The system $\gamma$ satisfies the measure
separation condition in $K$.
Let $\varphi: K \to K$ be defined by
\[
  \varphi (y) = \begin{cases}
                 2y & 0 \leq y \leq \frac{1}{2}, \\
                2y + 1 & \frac{1}{2} < y \leq 1.
                \end{cases}
\]
Then $\gamma_1$ and $\gamma_2$ are inverse branches of $\varphi$.
By Theorem \ref{thm:main},
the C$^*$-algebra $\mathcal{MC}_{\gamma_1, \gamma_2}$
is isomorphic to the Cuntz algebra $\mathcal{O}_2$.
The function $\varphi$ is not continuous. Note that
$\varphi$ in Theorem \ref{thm:main} is not necessary continuous.
\end{ex}

\begin{ex}
Let $K = [0,1]$ and
\[
  \gamma_1 (x) = \frac{1}{2} x, \quad  \gamma_2 (x) = - \frac{1}{2} x + 1.
\]
Then $K$ is the self-similar set with respect to
$\gamma = (\gamma_1, \gamma_2)$. 
The Hutchinson measure $\mu^H$ on $K$ coincides with
the Lebesgue measure on $K$. 
The system $\gamma$ satisfies the measure
separation condition in $K$.
Let $\varphi: K \to K$ be defined by
\[
  \varphi (y) = \begin{cases}
                 2y & 0 \leq y \leq \frac{1}{2}, \\
                 -2y + 2 & \frac{1}{2} \leq x \leq 1.
                \end{cases}
\]
The function $\varphi$ is called the tent map.
Then $\gamma_1$ and $\gamma_2$ are inverse branches of $\varphi$.
By Theorem \ref{thm:main},
the C$^*$-algebra $\mathcal{MC}_{\gamma_1, \gamma_2}$
is isomorphic to the Cuntz algebra $\mathcal{O}_2$.
This example is also considered in \cite{H3}.
\end{ex}

\begin{ack}
The author wishes to express his thanks to Professor Yasuo Watatani
for suggesting the problem and for many stimulating conversations.
\end{ack}

\end{document}